\documentclass[12pt]{amsart}
\usepackage[margin=1in]{geometry}                % See geometry.pdf to learn the layout options. There are lots.
\usepackage{graphicx}
\usepackage{comment}

\usepackage{amsmath}
\usepackage{amssymb}
\usepackage{amsfonts}
\usepackage{amsthm}
\usepackage{mathrsfs}
\usepackage{enumerate}
\usepackage{float}
\usepackage{fancyhdr}
\usepackage{xcolor}
\usepackage{hyperref}

 \usepackage[style=alphabetic, maxbibnames=9]{biblatex}
\newtheorem{theorem}{Theorem}[section]
\newtheorem{proposition}[theorem]{Proposition}
\newtheorem{lemma}[theorem]{Lemma}
\newtheorem{definition}[theorem]{Definition}
\newtheorem{corollary}[theorem]{Corollary}
\newtheorem{conjecture}[theorem]{Conjecture}

\newtheorem{remark}[theorem]{Remark}
\newtheorem{notation}[theorem]{Notation}

\numberwithin{equation}{section}

\newcommand{\R}{\mathbb{R}}

\begin{document}
%\pagestyle{fancy}
%\fancyhead{}
%\lhead{Bright-Marshall-Senger}
%\rhead{Fractal Dot Products}
\title[Pinned Dot Products]{Pinned Dot Product Set Estimates}

\author[P.~Bright]{Paige Bright}
\address{Department of Mathematics \\ University of British Columbia, 1984 Mathematics Road \\ Vancouver, BC, Canada V6T 1Z2}
\email{paigeb@math.ubc.edu}

\author[C.~Marshall]{Caleb Marshall}
\address{Department of Mathematics \\ University of British Columbia, 1984 Mathematics Road \\ Vancouver, BC, Canada V6T 1Z2}
\email{cmarshall@math.ubc.ca}

\author[S.~Senger]{Steven Senger}
\address{Department of Mathematics \\ Missouri State University, 901 S National Avenue \\ Springfield, MO, USA, 65897}
\email{stevensenger@missouristate.edu}

\date{\today}

\begin{abstract}
We study a variant of the Falconer distance problem for dot products. In particular, for fractal subsets $A\subset \R^n$ and $a,x\in \R^n$, we study sets of the form
\[
\Pi_x^a(A) := \{\alpha \in \R : (a-x)\cdot y= \alpha, \text{ for some $y\in A$}\}.
\]
We discuss some of what is already known to give a picture of the current state of the art, as well as prove some new results and special cases. We obtain lower bounds on the Hausdorff dimension of $A$ to guarantee that $\Pi^a_x(A)$ is large in some quantitative sense for some $a\in A$ (i.e. $\Pi_x^a(A)$ has large Hausdorff dimension, positive measure, or nonempty interior). Our approach to all three senses of ``size'' is the same, and we make use of both classical and recent results on projection theory.
\end{abstract}

\maketitle

%\tableofcontents

\section{Introduction}

We begin by discussing the problem that motivated this study. In 1985, Falconer \cite{Falconer_1985} introduced a classical question in geometric measure theory which is now commonly referred to as the Falconer distance problem. The problem asks for a threshold value $s$ so that any compact set $A\subset \R^n$ whose Huasdorff dimension is greater than $s$ will have the property that the set of distances between pairs of points in $A$ has positive Lebesgue measure. Define the \textit{distance set} of $A$ to be the set 
\[
\Delta(A) := \{|x-y| : x,y\in A\}.
\]
More precisely, Falconer conjectured the following.

\begin{conjecture}[Falconer]
    Let $A\subset \R^n$ with $n\geq 2$ be a compact set. Then
    \[
    \dim (A) > \frac{n}{2} \quad \implies \quad \mathcal H^1(\Delta(A)) >0.
    \]
    Here and throughout the paper, $\dim(\cdot)$ denotes Hausdorff dimension and $\mathcal H^1(\cdot)$ denotes the 1-dimensional Hausdorff measure.
\end{conjecture}

The study of this conjecture has utilized (and motivated work on) a number of cutting edge mathematical tools. In some way, it has served as a measuring stick for how much we can say about the rich relationships between measure theory, Fourier analysis and geometry. Early investigations by Mattila \cite{mattila1985}, Wolff \cite{wolff1999}, Bourgain \cite{bourgain2003}, and Erdo\u gan \cite{erdogan2006}, have since been built upon using techniques such as Kolmogorov complexity \cite{altaf2023distancesetsboundspolyhedral}, Fourier decoupling \cite{GuthFalconerPlane, du2023weightedrefineddecouplingestimates}, and projection theory \cite{ren2023discretized}. The current best known lower bounds towards the conjectures give:
\[
\begin{cases}
    \frac{5}{4}, & n=2 \quad \quad \text{(Guth--Iosevich--Ou--Wang \cite{GuthFalconerPlane})} \\
    \frac{n}{2} + \frac{1}{4} - \frac{1}{8n + 4}, & n\geq 3 \quad \quad \text{(Du--Ou--Ren--Zhang \cite{du2023newimprovementfalconerdistance})}.
\end{cases}
\]
We note here that both results above are proven utilizing results from projection theory; Guth--Iosevich--Ou--Wang utilize a radial projection result of Orponen \cite{Orponen_2019} and  Du--Ou--Ren--Zhang utilize Ren's discretized radial projection theorem \cite{ren2023discretized}. 

Both of the previous results for the Falconer distance set problem are a consequence of a \textit{pinned} result. Both Guth--Ou--Wang--Zhang and Du--Ou--Ren--Zhang show that in the range of $\dim A$ stated above, there exists an $a\in A$ so that the \textit{pinned distance set} pinned at $a$,
$$
\Delta^a(A) := \{|a - y| : y\in A\},
$$
has positive measure (which thus clearly implies the unpinned statement above). Finding lower bounds on the Hausdorff dimension of $A$ such that there exists a ``good pin'' $a\in A$ such that $|\Delta^a(A)| >0$ is sometimes referred to as the pinned distance set problem.

Motivated by this pinned distance set problem, as well as by recent results in projection theory (such as those involving radial projections \cite{orponen2024kaufman, ren2023discretized, bright2024radialprojections}, as well as estimates involving orthogonal projections  \cite{pramanikyangzahl, zahl2023maximalfunctionsassociatedfamilies, ren2023furstenberg}), we study \textit{pinned dot product sets} over Euclidean space. Given some $A\subset \R^n$ and $a,x\in \R^n$, we define the set of \textit{dot products of $A$ pinned at $a$ and translated by $x$} as
\[
\Pi_x^a(A):= \{\alpha \in \R : (a-x)\cdot y = \alpha, \text{ for some $y\in A$}\}.
\]
If $x$ is the origin, this is simply the set of \textit{pinned dot products of $A$ pinned at $a$}. 

The dot product analog of the Falconer distance problem asks: how large does $\dim A$ need to be so that $\Pi_x^a(A) \subset \R$ is large in some quantitative sense? However, unlike in the pinned distance set problem, we also introduce the translation parameter $x$. This allows us to prove results both for general $x \in \mathbb{R}^n$ (i.e. a translation invariant result) as well as results over ``nice-enough'' sets of translations $X_A \subset \mathbb{R}^n$ (i.e. translation dependent results). What follows will be statements of results that are either translation invariant or translation dependent, some with additional assumptions, given alongside known results to provide a survey of what is known about the dot product variant of the Falconer problem.

\subsection{Translation invariant results} Our first main result gives an answer to this question regardless of choice of $x\in \R^n$. Note that we say $A'\subset A$ is \textit{full dimensional} if $\dim A' = \dim A$.

\medskip

\begin{theorem} \label{thm-main}
    Let $A\subset \R^n$ Borel ($n\geq 2$) with $\dim A = s$.
    \medskip
    
    \begin{enumerate}
        \item Let $0< u \leq 1$ and suppose $s > \frac{n+u}{2}$ for $n\geq 2$. Then, for all $x\in \R^n$, there exists a full dimensional Borel subset $A_x\subset A$ such that 
        \[
        \dim \Pi_x^a(A) \geq u \text{ for all $a\in A_x$.}
        \]
        \item Suppose $s>\frac{n+1}{2}.$ Then, for all $x\in \R^n$, there exists a full dimensional Borel subset $A_x\subset A$ such that 
        \[
        \mathcal H^1(\Pi_x^a(A)) >0, \text{ for all $a\in A_x$}.
        \]
        \item Suppose $s > \frac{n+2}{2}$ for $n\geq 3.$ Then, for all $x\in \R^n$, there exists a full dimensional Borel subset $A_x\subset A$ such that
        \[
        \Pi_x^a(A) \text{ has nonempty interior for all $a\in A_x$}.
        \]
    \end{enumerate}
\end{theorem}

As far as the authors know, Theorem \ref{thm-main}.(1) and (3) 
are new to the literature. The statement (1) can be seen as an analog to the distance results of Oberlin--Oberlin in \cite{OO}, while the statement (3) is a Mattila--Sj\"olin type result, as seen in \cite{mattilaSjolin}. Theorem \ref{thm-main}.(2), on the other hand, has been known has been known in the literature since 2016 due to the more general result of Iosevich--Taylor--Uriarte-Tuero \cite{iosevich2016pinned}. In fact, their result applies for a wide class of functionals. Specifically, for a general functional $\phi(x,y)$ in place of the dot product or distance, they proved that result analogous to Theorem \ref{thm-main}.(2) holds if $\phi$ is continuous and infinitely differentiable almost everywhere on $A\times A,$ and has the property that the Monge-Ampere determinant
\[\det\left[\begin{tabular}{c c} $0$ & $\nabla_x\phi$\\ $-(\nabla_y\phi)^T$ & $\frac{\partial^2\phi}{\partial x_i \partial y_j}$\end{tabular}\right]\]
does not vanish on any of the non-zero level sets of $\phi,$ which, for $t\neq0$ are given by $\{(x,y):\phi(x,y)=t\}.$ In fact, their proof shows even more, but we do not detail it here. Because we are focusing on the case where $\phi(x,y) = x \cdot y,$ we use classical ideas from projection theory for radial and orthogonal projections. This is a natural tool to consider, as we recall that one of the primary goals of Iosevich--Taylor--Uriarte-Tuero was to generalize some work of Peres and Schlag \cite{pxxschlag2000}, whose original method of proof relied on non-linear projections. These tools are discussed in Section \ref{sec:method}.  

% For example, the breakthrough paper of Ren and Wang \cite{ren2023furstenberg} proves a sharp estimate for orthogonal projections in the plane. Using their result, one immediately obtains the following.

% \begin{proposition}\label{thm:mainrenwang}
%     Let $\frac{1}{2}\leq u \leq 1$. Suppose $A\subset \R^2$ Borel with $\dim A> \frac{2u+1}{2}$. Then there exists a full dimensional subset $X_A\subset X$ such that, for each $x\in X_A$ there exists a full dimensional subset $A_x\subset A$ such that 
%     \[
%     \dim (\Pi_x^a(A)) \geq u , \text{ for all $a\in A_x$.}
%     \]
% \end{proposition}

% \begin{remark}
%     \rm{We assume that $u \geq \frac{1}{2}$ so that $\dim A>1$. This is so that $\dim \pi_x(A) \geq \dim A -1 >0$ for all $x\in \R^n$.}
% \end{remark}

One benefit of the method of projections is that it highlights a general heuristic that ``better projection estimates imply better dot product results.'' This provides a certain immediacy of improvement of the dimensional lower bounds of Theorem \ref{thm-main}, under the assumption our set $A \subset \mathbb{R}^n$ has additional structure, as has been observed in other settings, such as in Hart--Iosevich--Koh--Rudnev \cite{HIKR}.
For example, we can obtain easy improvements assuming Fourier analytic structure of our set $A$. If $A \subset \mathbb{R}^n$, is Borel, let $\mathcal{M}(A)$ denote the space of Borel probability measures supported on $A$. We then define the \textit{Fourier dimension} of $A$ as
$$
\dim_F A : = \sup \{s : \exists \, \mu \in \mathcal{M}(A) \textrm{ with } \vert \hat{\mu} (\xi) \vert \lesssim \vert \xi \vert^{-s/2}, \forall \, \xi \in \mathbb{R}^n \}.
$$
We now state a basic property of Fourier dimension that we will use later. A proof can be found in \cite{mattila1999}, Section 12.17.
\begin{proposition}\label{fDimProp}
For any $A \subset \mathbb{R}^n$,
\[\dim_F A \leq \dim A.\]
\end{proposition}

The following result shows a distinct connection between Fourier dimension and the Hausdorff dimension of the dot product set.
\begin{proposition}\label{prop:fourierdimension}
Suppose that $A \subset \mathbb{R}^n$ satisfies $\dim A > 1$ and $\dim_{F} A \geq u$ for some $u \in (0,1]$. Then, for every $a,x \in \mathbb{R}^n$ with $x\neq a$, one necessarily has that
$$
\dim \big(\Pi_x^a (A) \big) \geq u.
$$
\end{proposition}

Both results stated above follow almost immediately from the method of proof for Theorem \ref{thm-main}, once we utilize a well-known result regarding the orthogonal projection of sets with positive Fourier dimension.

So far, we have focused on dot products determined by pairs of points. However, variants where multiple dot products are considered between more than two points have a rich history, as in, for example, \cite{iosevich2016pinned}. As an application of some of these and related results, we follow the work of Ou and Taylor \cite{OT2020} and prove analogous results for tree configurations in Section \ref{sec:trees}.

\subsection{Translation dependent results}
While Theorem \ref{thm-main} concerns dot product sets for \textit{any} translation $x \in \mathbb{R}^n$, we can also obtain a version of these estimates with smaller dimensional bounds on $\dim A$ if we only seek out good bounds for \textit{some} $x\in \R^n$. As such, we say such improvements on lower bound is \textit{translation dependent}.

The types of translation sets $X \subset \mathbb{R}^n$  which allow us to improve upon Theorem \ref{thm-main} possess a certain necessary geometry. This geometry is dictated by our method, which applies the work of Orponen--Shmerkin--Wang \cite{orponen2024kaufman} and Ren \cite{ren2023discretized}. Both of these results obtain improved dimensional estimates for radial projections for sets that have a non-concentration property we call $k$-\textit{planar dispersion}. 

\begin{definition}
     Let $X \subset \mathbb{R}^n$ Borel and $k \in \{1,...,n-1\}$. Then, $X$ has $k$-\textbf{planar dispersion} if $\dim (X \setminus P^k) = \dim X$ for all $k$-planes $P^k \subset \mathbb{R}^n$.
\end{definition}

Our second main result concerns this translated variant of Theorem \ref{thm-main}.

\begin{theorem}\label{thm:main-translation}
Let $n \geq 2$ and let $k \in \{1,...,n-1\}$ be some parameter to be chosen below. Suppose that $A, X \subset \mathbb{R}^n$ Borel with $\dim X \geq \dim A = s$ and further suppose $X$ has $k$-planar dispersion. Then, 

\medskip 

\begin{enumerate}
    \item Let $n\geq 3$ and $0 < u \leq 1$. Suppose $s > \frac{n+u -1}{2}$ and let $k = \lceil \frac{n+u-1}{2} \rceil$. Then, there exists a full dimensional subset $X_A \subset X$ such that, for each $x \in X_A$ there exists a full dimensional subset $A_x\subset A$ such that 
    \begin{equation*}
        \dim \big(\Pi_x^a (A)\big) \geq u, \textrm{ for all } a \in A_x.
    \end{equation*}

    \item Suppose $s > \frac{n}{2}$. Let $k = \lceil \frac{n}{2} \rceil$. Then, there exists a full dimensional subset $X_A \subset X$ such that, for each $x \in X_A$ there exists a full dimensional subset $A_x\subset A$ such that
    \begin{equation*}
        \mathcal{H}^1 \big(\Pi_x^a (A)\big) > 0, \textrm{ for all } a \in A_x.
    \end{equation*}
    
    \item Let $n \geq 3$ and suppose $s > \frac{n+1}{2}$. Let $k = \lceil \frac{n+1}{2} \rceil$. Then, there exists a full dimensional subset $X_A \subset X$ such that, for each $x \in X_A$ there exists a full dimensional subset $A_x\subset A$ such that
     \[
        \Pi_x^a(A) \text{ has nonempty interior for all $a\in A_x$}.
        \]
\end{enumerate}
\end{theorem}

Comparing with Theorem \ref{thm-main}, one notices an immediate improvement in the necessary dimensional lower bound in Theorem \ref{thm:main-translation} for our set $A \subset \mathbb{R}^n$. Notably, Theorem \ref{thm:main-translation} matches the lower bound of the Falconer distance conjecture. Consequently, we conjecture the following.

\begin{conjecture}
    If $A\subset \R^n$ is Borel and $\dim A>\frac{n}{2}$, then there exists a pin $a\in A$ so that 
    \begin{equation}\label{eq:sharpdotproduct}
    \mathcal H^1(\Pi_0^a(A)) >0.
    \end{equation}
\end{conjecture}

Some motivation for the above Conjecture is given by Theorem 
\ref{thm:main-translation}.$(2)$, as well as the fact that the dimensional threshold $\frac{n}{2}$ is sharp when $n = 2$, as can be seen by taking $A = \mathcal{C} \times \{0\}$, where $\mathcal{C}$ satisfies $\dim \mathcal{C} = 1$ but $\mathcal{H}^1 (\mathcal{C}) = 0$. Such sets are well-known in the literature. 

\begin{remark}
\rm{
We assume $\dim X\geq \dim A$ primarily for the cleanliness of the statement and proof. If one wanted to obtain statements akin to Theorem \ref{thm:main-translation} for $\dim X< \dim A$, the flexibility of our approach allows for this. However, this introduces more dependencies between $k$, $\dim A$, and $\dim X$. As such, we focus on the cleaner statement of Theorem \ref{thm:main-translation}.}
\end{remark}

% Similar to the previous section, we obtain a version of Theorem \ref{thm:main-translation}.(1) using the exceptional set estimate of Ren--Wang \cite{ren2023furstenberg} in the plane.

% \begin{proposition}\label{thm:mainrenwang-translation}
%     Let $0\leq u \leq 1$. Suppose $A,X\subset \R^2$ Borel with $dim X\geq \dim A> \frac{u}{2}$. Then there exists a full dimensional subset $X_A\subset X$ such that, for each $x\in X_A$ there exists a full dimensional subset $A_x\subset A$ such that 
%     \[
%     \dim (\Pi_x^a(A)) \geq u , \text{ for all $a\in A_x$.}
%     \]
% \end{proposition}

One intuitive reason for the improvement in Theorem \ref{thm:main-translation} is that the set of translated and pinned dot products can be written as
\begin{equation}\label{eq:differencedots}
\Pi_x^a (A) = \{(a - x) \cdot y : y \in A \} = \Pi_0^a (A) - \Pi_0^x (A),
\end{equation}
where ``$-$" here denotes the difference set $X - Y : = \{x - y : x \in X , y \in Y\}$. It is well-known that two sets $X, Y$ can jointly have zero Hausdorff dimension or null Lebesgue measure, or empty interior---and yet their difference set $X - Y$ could, perhaps, satisfy all three. Indeed, a classic example of this phenomenon is given in \cite{falconertextbook} Chapter 7, which provides two sets $X, Y \subset \mathbb{R}$ which satisfy $\dim X = \dim Y = 0$, and yet whose difference $X - Y$ is the entire unit interval. 

Theorem \ref{thm:main-translation} is then asserting that, so long as one selects a large enough set of translations $X \subset \mathbb{R}^n$, the right-hand side difference in \eqref{eq:differencedots} is guaranteed to be larger (in the appropriate dimensional, measure, or interior sense) than that which is given by the singleton translation result of Theorem \ref{thm-main}. Of course, when $\dim A$ is large enough, Theorem \ref{thm-main} is always stronger than Theorem \ref{thm:main-translation}. However, the importance of introducing the translation set $X$ is in obtaining dot product type estimate for these critical values of $\dim A$. 
To illustrate this idea more directly, we have the following Corollary of Theorem \ref{thm:main-translation}, which gives another perspective of why one may want a translated result.

\begin{corollary}\label{cor:additive}
    Let $n \geq 2$, let $k = \lceil \frac{n}{2}\rceil$, and suppose $A\subset \R^n$ has $k$-planar dispersion and satisfies $\dim A>\frac{n}{2}$. Then, there exist two full-dimensional subsets $A_1,A_2\subset A$ such that 
    \begin{equation}\label{eq:additive1}
    \mathcal H^1\big(\Pi_{a_1}^{a_2}(A)\big) >0, \text{ for all } (a_1,a_2) \in A_1\times A_2.
    \end{equation}
    In particular, there exists a set $A_0 \subset A - A$ with $\dim A_0 = \dim A$ such that
    $$
    \mathcal{H}^1 \big( \Pi_0^{a_0} (A)\big) > 0, \textrm{ for every } a_0 \in A_0.
    $$
    
    %In particular, if $A \subset \mathbb{R}^n$ satisfies the above assumptions, and also satisfies the additive condition $A - A \subset A$, then there necessarily exists a set $A_0 \subset A$ of full dimension such that
    %\begin{equation}\label{eq:additive2}
    %\mathcal{H}^1 \big(\Pi_0^{a} (A) \big) > 0, \quad \forall a \in A_0.
    %\end{equation}
\end{corollary}

\begin{proof}[Proof of Corollary \ref{cor:additive}]
    To prove \eqref{eq:additive1}, take $X = A$ in Theorem \ref{thm:main-translation}, which produces a full-dimensional set of translates $X_A := A_1 \subset A$, and a full dimensional set of pins $A_x := A_2 \subset A$ for which inequality \eqref{eq:additive1} must hold. %We then obtain the inequality \eqref{eq:additive2} by noticing that:
    %\begin{eqnarray*}
    %\Pi_{a_1}^{a_2} (A) & : = & \{\alpha \in \mathbb{R} :(a_1 - a_2) \cdot a = \alpha \textrm{ for some } a \in A \} \\
    %\quad & = & \Pi_0^{a_1 - a_2}(A).
    %\end{eqnarray*}
    %Theorem \ref{thm:main-translation} guarantees that $\dim A_1 = \dim A_2 = \dim A$ and we assume that $A - A \subset A$, proving \eqref{eq:additive2}.
\end{proof}

This idea of examining the dot product along \textit{difference sets} has been studied in the literature, such as in \cite{GIT2019}. A direct consequence of the above Corollary is that, if $n$ is even, $A$ must automatically have $k$-planar dispersion, which is a consequence of the inequality $\dim A >\frac{n}{2}$. We remark that a similar version of the above Corollary holds if we instead require that $\Pi_{a_1}^{a_2} (A)$ have non-empty interior or positive Hausdorff dimension.

Another example which improves upon Theorem \ref{thm:main-translation}---this time assuming additional \textit{geometric} constraints on $A$ and a specific translation $x$---is obtained using restricted projections. The idea comes from the notion of a \textit{restricted} exceptional set in projection theory, which is a certain object that only considers $\theta \in \mathbb{S}^{n-1}$ lying in the image of some non-degenerate curve. By \textit{non-degenerate curve}, we mean a smooth map $\gamma : [0,1]\to \mathbb{S}^{n-1}$ such that for every $t \in [0,1]$, one has
\[
\det [\gamma(t), \gamma^{(1)}(t), \cdots , \gamma^{(n-1)}(t)] \neq 0.
\]
Using a recent restricted projections estimate of Zahl \cite{zahl2023maximalfunctionsassociatedfamilies}, we obtain the following. We let $\pi_x:\R^n \setminus \{x\}\to \mathbb{S}^{n-1}$ be the \textit{radial projection} onto $x$ defined via $\pi_x(y) := \frac{y-x}{|y-x|}$.

\begin{proposition}\label{prop:restricted}
    Let $A\subset \R^n$ and $0< u\leq \min\{\dim A,1\}$. Suppose, for some $x \in \mathbb{R}^n$, there exists a non-degenerate curve $\gamma: [0,1]\to \mathbb{S}^{n-1}$ such that 
    \[
    \dim \{t \in [0,1] : \gamma(t) \in \pi_x(A)\} >u.
    \]
    Then, there exists a subset $A_x\subset A$ such that $\dim A_x \geq u$ and $\dim \Pi_x^a(A) \geq u$ for all $a\in A_x.$
\end{proposition}

In particular, this gives another geometric classification for sets which admit (at least one) good pin $a \in A$ when translated by $x \in \mathbb{R}^n$. 

\subsection{Paper outline}

We begin with a discussion of how one turns the pinned dot products problem into one about radial and orthogonal projections in Section \ref{sec:projbackground}. We present a fairly general geometric methodology in \ref{sec:method} that we believe may apply to other types of point-configuration problems. The projection theoretic tools we will need throughout the paper are are presented in Section \ref{sec:projbackgroundsTOTAL}. Proofs of the main results stated above are contained in Section \ref{sec:proofs}. Finally, we conclude with applications and constructions of sharpness examples for multiple point configurations in Section \ref{sec:trees}. 

\bigskip
\begin{sloppypar}
\noindent {\bf Acknowledgements.} We want to thank Alex Iosevich, Eyvi Palsson, Pablo Shmerkin, and Josh Zahl for their insightful discussion and encouragement to continue trying to improve our results.
\end{sloppypar}

\section{From dot products to projection theory}
\label{sec:projbackground}
This section introduces our geometric framework for studying the dot product sets $\Pi_x^a (A)$, then utilizes said framework to relate the size of $\Pi_x^a (A)$ to the orthogonal projection of $A$. We also introduce the reader to the projection theoretic estimates we will need for our results.

\subsection{The geometry of dot products}\label{sec:method}
The key geometric insight to our projection theoretic approach is the following. Let $a,x\in \R^n$, with $a\neq x$. Given any $\alpha \in \Pi_x^a(A)$, the set of
$y$ satisfying $a\cdot (y-x) = \alpha$, is a hyperplane with normal direction $\pi_x(a) :=\frac{a-x}{|a-x|}$. We use $U_x^a(\alpha)$ to denote this hyperplane. This observation allows us to reduce our problem to understanding how large this set of planes must be. To answer this we can utilize orthogonal projections, as there is a bijection between the set of hyperplanes $\{U_x^a(\alpha)\}_{\alpha\in \Pi_x^a(A)}$ and the set $P_{\pi_x(a)}(A)$, where $P_\theta$ is the orthogonal projection onto the line through the origin in the direction of $\theta$.

\begin{figure}[ht]
    \centering
    \includegraphics[scale=.5]{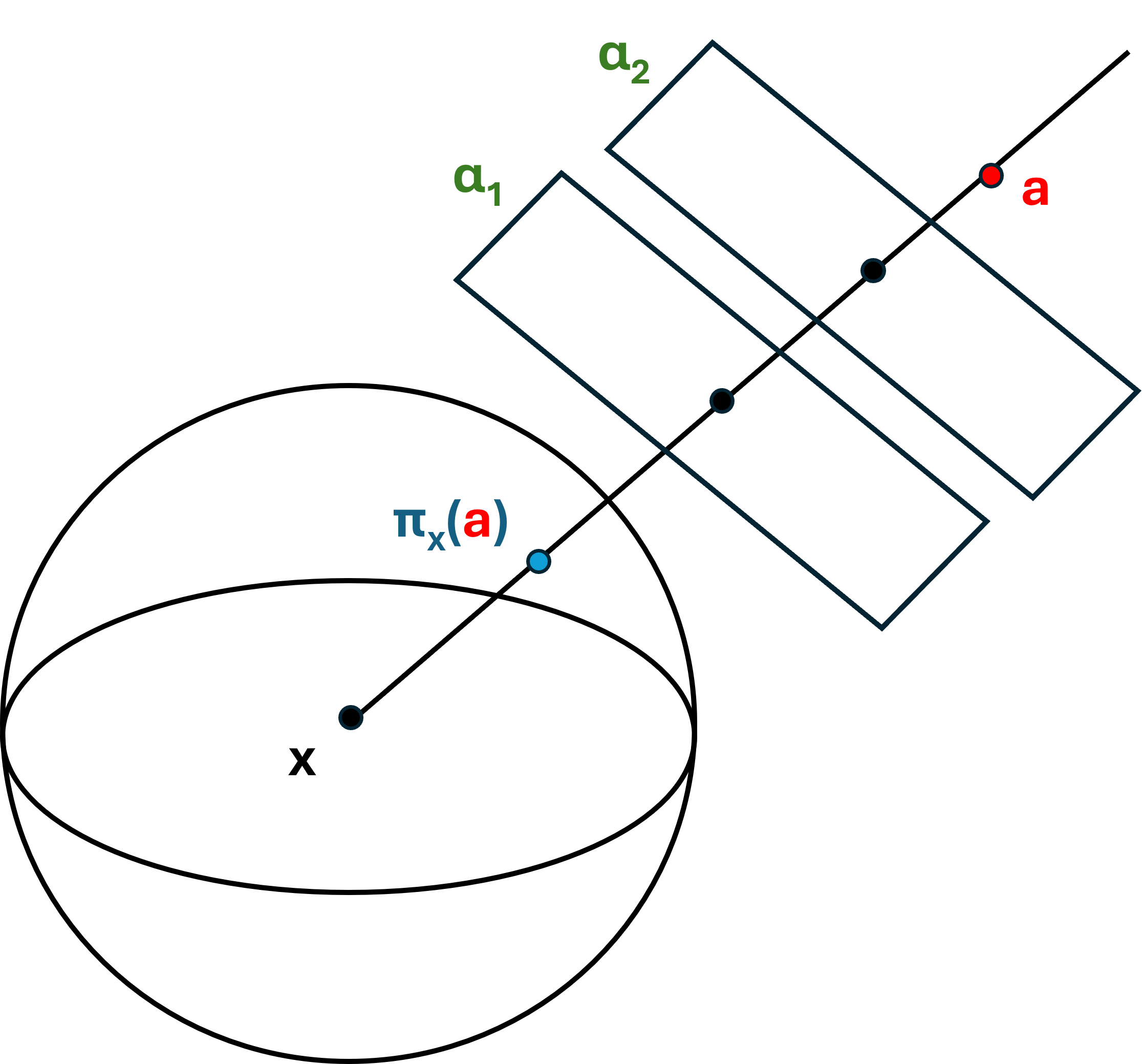}
    \caption{Two hyperplanes corresponding to elements $\alpha_1,\alpha_2 \in \Pi_x^a(A)$.}
\end{figure}

Using the existence of this bijection between the set of (translated) pinned dot products and the set $P_\theta(A)$, we next show that the ``size'' of both these sets are the same (our Key Lemma \ref{lma-keylemma}). This establishes that, so long as $a\neq x$, we have $\mathcal H^1(\Pi_x^a(A)) = 0$ if and only if $\mathcal H^1(P_{\pi_x(a)}(A)) = 0$. This leads us to consider the exceptional set of orthogonal projections
\[
``E(A) = \{\theta \in\mathbb{S}^{n-1}: P_\theta(A) \text{ ``is small''\}.''}
\]
We will define what it means for $P_\theta(A)$ to be small later. This allows us to apply \textit{exceptional set estimates}: results in projection theory giving upper bounds on $\dim E(A)$.

\begin{remark}
    \rm{Note that in some parts of the literature, $E(A)$ is viewed as the set of \textit{lines} through the origin in \textit{direction} $\theta$ such that $P_\theta(A)$ is small. However, one can identify such lines with their direction as is done in \cite{mattila1999}. In particular, this identification preserves the Hausdorff dimension which is the key property we need for our results.}
\end{remark}

Utilizing the exceptional set $E(A)$, the pinned dot product problem turns into the following question: is there a direction in the set of \textit{possible} directions to orthogonally project onto, $\pi_x(A)$, that does not lie in the set of \textit{bad} directions, $E(A)$? Using standard results on $\dim \pi_x(A)$ and $\dim E(A)$, we obtain that there are in fact many directions in $\pi_x(A)$ \textit{not} contained in $E(A)$ if $\dim A$ is large enough. This framework will establish Theorem \ref{thm-main}. Then, using recent results concerning the dimension of $\pi_x(A)$, due to Orponen--Shmerkin--Wang \cite{orponen2024kaufman} and Ren \cite{ren2023discretized}, our approach yields a translated and pinned dot product result for a more general range of $\dim A$. This will establish Theorem \ref{thm:main-translation}.

\subsection{Radial and orthogonal projections} \label{sec:projbackgroundsTOTAL}

Our proof methods rely on some standard, as well as some recent, results in projection theory for orthogonal and radial projections. For each $\theta \in \mathbb{S}^{n-1}$, let $\ell_\theta$ be a line through the origin that passes through $\theta$. Define the orthogonal projection $P_\theta: \R^n\to \ell_\theta$ as
\[
P_\theta(y) := (y\cdot \theta) \theta, \text{ for all } y\in \R^n.
\]
One rich and well-connected topic in geometric measure theory is the study of the dimensional stability of the Hausdorff dimension of a set $A$ under these orthogonal projection maps.

For our purposes in particular, we need upper bounds on the dimension of the set of directions whose orthogonal projections have small dimension, null measure, or empty interior. Standard results in projection theory heuristically state that such directions rarely occur and are thus \textit{exceptional} (see \cite[Section 4]{Mattila_2015} for statements and proofs of many of these sorts of results). Hence, we refer to the aforementioned bounds on the set of directions whose orthogonal projections are small as \textit{exceptional set estimates}.

We introduce some notation which makes the above discussion more precise.

\begin{notation}
    If $A\subset \R^n$ is Borel and $0< u \leq 1$, we define three types of exceptional sets for orthogonal projections:
    \begin{align*}
        E_u(A) &:= \{\theta \in \mathbb{S}^{n-1} : \dim (P_\theta(A))< u\} \\
        E_+(A) &:= \{\theta \in \mathbb{S}^{n-1} : \mathcal H^1(P_\theta(A)) = 0\} \\
        E_\circ(A) &:= \{\theta \in \mathbb{S}^{n-1}: P_\theta(A) \text{ has empty interior}\}.
    \end{align*}
    For the use of discussion, we refer to a generic one of these three choices as simply $E(A)$.
\end{notation}

The exceptional set estimates for orthogonal projections needed to prove Theorem \ref{thm-main} and \ref{thm:main-translation} are the following.

\begin{proposition}\label{prop-measure_ESE}
    Suppose $A\subset \mathbb{R}^n$ is Borel. Then, one has the estimates
    \begin{enumerate}
        \item for all $0 < u \leq 1$, one has $\dim E_u(A) \leq \max\{n-1 + u - \dim A,0\}$;

        \medskip
        
        \item if one assumes that $\dim A >1$, then $\dim E_+(A) \leq n-\dim A$;

        \medskip
        
        \item if one further assumes that $n \geq 3$ and $\dim A > 2$, then  $\dim E_\circ(A)  \leq n+1- \dim A$.
    \end{enumerate}
\end{proposition}

Proposition \ref{prop-measure_ESE}.(1) was proved by Kaufman \cite{kaufman1968hausdorff} when $n = 2$, and Mattila \cite{mattila1975exceptional} when $n \geq 3$. Proposition \ref{prop-measure_ESE}.(2) is due to Falconer \cite{falconer1982hausdorff}. Proposition \ref{prop-measure_ESE}.(3) was proved Peres and Schlag \cite{pxxschlag2000}. All three estimates are proved in \cite[Chapter 5]{Mattila_2015} using the notion of the Sobolev dimension of measures. It is worth noting that the Proposition \ref{prop-measure_ESE}.(1) is generally not believed to be sharp. Indeed, very recently, a sharp bound on $\dim E_u(A)$ was obtained by Ren and Wang \cite{ren2023furstenberg} when $n = 2$. 

%We will discuss this improvement later, and show that it implies an immediate strengthening of Theorem \ref{thm-main}.(1).

\begin{remark}
    \rm{Importantly, Ren and Wang \cite{ren2023furstenberg} recently improved Proposition 2.3.$(1)$ when $n = 2$ to the sharp estimate $\dim E_u(A) \leq \max\{2u-\dim A, 0\}$. This improvement would allow us to improve the dimensional threshold of Theorem \ref{thm-main}.$(1)$ and Theorem \ref{thm:main-translation}.$(1)$ when $n = 2$. We leave the details for the interested reader.
}
\end{remark}

We now turn our focus to results regarding radial projections. Given $x\in \R^n$, we define the radial projection $\pi_x: \R^n\setminus \{x\} \to \mathbb{S}^{n-1}$ via
\[
\pi_x(y) := \frac{y-x}{|y-x|}, \text{ for all }y\in \R^n \setminus \{x\}.
\]

\begin{remark}
    \rm{Given $A\subset \R^n$, by abuse of notation we denote
    \[
    \pi_x(A):=\pi_x(A\setminus \{x\}) .
    \]}
\end{remark}

For our purposes, we make use of two radial projection estimates---one for Theorem \ref{thm-main}, and another for Theorem \ref{thm:main-translation}.

\begin{proposition}\label{prop-zeromeasureradproj}
    Suppose $A\subset \R^n$ Borel with $\dim A >1$. Then, for all $x\in \R^n$,
    \[
    \dim \pi_x(A) \geq \dim A -1.
    \]
\end{proposition}

\begin{proof}
    There are two proofs of the above theorem. One follows via a Frostman measure argument (see \cite[Section 2]{bright2024radialprojections}), and the other follows via a covering argument. We present the covering argument here.

    Let $\hat{A}_x : = \{\vert a-x \vert : a \in A \}$. Then, it is clear that $$A \subset \bigg(\pi_x (A\setminus \{x\}) \times \hat{A}_x \bigg).$$
    In words, $A$ is contained in the product set obtained by pairing every direction produced by $A$ and $x$ with every modulus produced by $A$ and $x$. In essence, we are putting $A$ in polar coordinates. The remainder follows from a dimensional formula (see Proposition 8.10 in \cite{mattila1999}) for product sets, which is
    \begin{equation}\label{eq-dimprodset}
    \dim (X \times Y) \leq \dim X + \dim_P Y,
    \end{equation}
    where $\dim_P Y$ denotes the packing dimension of $Y$. Since, by definition, we have $\hat{A} \subset \mathbb{R}$, we have the trivial upper bound $\dim_P \hat{A}_x \leq 1$. Hence, applying \eqref{eq-dimprodset} with $X = \pi_x (A\setminus\{x\})$ and $Y = \hat{A}_x$, we obtain the result.
\end{proof}

Note that in general, we would expect the above lower bound to be better: given more vantage points $X\subset\R^n$, the more likely there will exist an $x\in X$ such that $\pi_x(A\setminus\{x\})$ is larger than the simple lower bound of $\dim A-1$. This is in fact what the work of Orponen--Shmerkin--Wang \cite{orponen2024kaufman} and Ren \cite{ren2023discretized} shows:

\begin{theorem}[Theorem 1.1, \cite{ren2023discretized}]\label{thm:renplanes}
    Let $X,A\subset \R^n$ Borel. If $X$ is not contained in a $k$-plane and $X\neq \emptyset$, then 
    \[
    \sup_{x\in X} \dim (\pi_x(A)) \geq \min\{\dim X, \dim A, k\}.
    \]
\end{theorem}

For our purposes, we will be using a corollary of Theorem \ref{thm:renplanes} that follows via the proof framework of Orponen--Shmerkin--Wang \cite{orponen2024kaufman}.

\begin{corollary}\label{cor:epsilonradproj}
    Let $\emptyset \neq A,X\subset \R^n$ Borel satisfy $\dim (X\setminus P^k) = \dim X$ for all $k$-planes $P^k\in \mathcal A(n,k)$ (i.e. $X$ has $k$-planar dispersion). Then, for every $\epsilon>0$, the exceptional set
    \[
    E_\epsilon:= \{x\in X: \dim \pi_x(A) \leq \min\{\dim X, \dim A, k\} - \epsilon \}
    \]
    satisfies $\dim (X\setminus E_\epsilon) = \dim X$.
\end{corollary}

\begin{proof}
    For the sake of contradiction, suppose $\dim (X\setminus E_\epsilon) < \dim X$. Then, $E_\epsilon$ cannot be contained in any $k$-plane as otherwise 
    \[
    \dim (X\setminus P^k) \leq \dim (X\setminus E_\epsilon) < \dim X
    \]
    which contradicts our hypothesis on $X$. Furthermore, it also implies $\dim E_\epsilon = \dim X$. Thus, applying Theorem \ref{thm:renplanes} to the pair $E_\epsilon, A$ gives 
    \[
    \sup_{x\in E_\epsilon} \pi_x(A) \geq \min\{\dim E_\epsilon, \dim A, k\} = \min\{\dim X, \dim A, k\}.
    \]
    Given $\epsilon>0$, this contradicts how we defined $E_\epsilon$.
\end{proof}

These are the key results involving radial and orthogonal projections that we will need for our proofs of the main theorems. The final result we need is an equality between the $\mathcal{H}^s$ measures of the set of translated and pinned dot products and the orthogonal projection. We state this precisely in the following result.

\begin{lemma}\label{lma-keylemma}
    Let $x\neq a\in \R^n$, $A\subset \R^n$ Borel and $0 \leq s \leq 1$. Then, there exists a constant $c_{a,x,s} > 0$ such that
    \begin{equation}\label{eq:keylemma}
        \mathcal H^s(\Pi_x^a(A)) = c_{a,x,s} \mathcal H^s(P_{\pi_x(a)}(A)),
        % = \mathcal{H}^s \bigg(\bigcup_{\alpha \in \Pi_x^a(A)} \ell_{\pi_x(a)}\cap U^a_x(\alpha) \bigg) =  \mathcal{H}^s \big( \mathcal{U}^a_x \big),
    \end{equation}
    where the $\mathcal{H}^s$ measure in the first term is taken on $\R$ and is over $\ell_{\pi_x (a)},$ the line in the direction of $\pi_x(a)$ (identified with $\mathbb{R}$) in the second.
\end{lemma}

\begin{proof}
The statement follows once we recognize that, since for any $y \in A$, we have:
\begin{equation*}
P_{\pi_x (a)} (y)  :=  [\pi_x (a) \cdot y] \pi_x (a) =  \bigg[\frac{1}{\vert a - x \vert}\big((a-x) \cdot y \big)\bigg] \pi_x (a).
\end{equation*}
In particular, since $\pi_x (a)$ is a unit vector, we have that
$$
\big\vert P_{\pi_x(a)} (y) \big\vert = \frac{\vert (a-x) \cdot y \vert}{\vert a - x \vert} 
$$
If we assume that $y \in A$, and identify the set $P_{\pi_x (a)} (A) \subset \ell_{\pi_x (a)}$ with the corresponding subset of $\mathbb{R}$ (as explained above) we see that
$$
\mathcal{H}^s \big(P_{\pi_x(a)} (A) \big) = \mathcal{H}^s \big(\vert a - x \vert^{-1} \, \Pi_x^a (A) \big) = |a - x|^{-s}\cdot  \, \mathcal{H}^s \big(\Pi_x^a (A) \big),
$$
which is \eqref{eq:keylemma} with $c_{a,x,s} = |a-x|^s$. 
\end{proof}

\section{Concluding the proofs}
\label{sec:proofs}

\subsection{Proofs of translation invariant results}

First, we prove Theorem \ref{thm-main}.(2). The proofs of (1) and (3) follow by essentially the same argument, but with the different definitions of $E(A)$ and by applying the corresponding bounds in Proposition \ref{prop-measure_ESE}.

\begin{proof}[Proof of Theorem \ref{thm-main}.(2)] 
Let $x \in \mathbb{R}^n$ and $A \subset \mathbb{R}^n$ be given with $\dim A = s > \frac{n+1}{2}$. Note that, by Proposition \ref{prop-zeromeasureradproj} and Proposition \ref{prop-measure_ESE}, we have
\begin{equation}
    \dim \pi_x \big(A \big) \geq s - 1 > \frac{n-1}{2}.
\end{equation}
Additionally,
\begin{equation}\label{ePlusBound}
    \dim E_{+} (A) \leq n - s < \frac{n - 1}{2}.
\end{equation}
In particular, we have $\dim \pi_x(A) > \dim E_+(A)$. So, defining the set
$
\overline{A_x} = \pi_x \big(A\big)
$
as
$$
\overline{A_x} : = \pi_x (A) \setminus E_{+}(A),
$$
we have $\dim \overline{A_x} = \dim \pi_x(A) \geq s-1.$ Now, define the set
$$
A_x : = \pi_x^{-1} \big(\overline{A_x})\cap A = \{a \in A \setminus \{x\} : \pi_x (a) \in \overline{A_x} \}.
$$
By definition of $E_{+}(A)$, we have
$$
\mathcal{H}^1 \big(P_{\pi_x (a)} (A) \big) > 0, \quad \forall a \in A_x.
$$
We now may apply the Key Lemma \ref{lma-keylemma} at each point $a \in A_x$, giving
$$
\mathcal{H}^1 \big(\Pi_x^a (A)\big) > 0, \quad \forall a \in A_x.
$$

It remains to show that $\dim A_x = \dim A$. Clearly, $\dim A_x \leq \dim A$. We continue by way of contradiction. To this end, let $A_* : = A \setminus A_x$ and suppose for the sake of contradiction that $\dim A_* = \dim A$. Since we have removed $A_x$, we must have that:
$$
\pi_x (A_*) \subset E_+(A).
$$
But then, recalling the definition, we would necessarily have
\[
A_* \subset E_{+}(A)= \{\theta \in \mathbb{S}^{n-1}: \mathcal H^1(P_{\theta}(A)) = 0\}.
\]
Using \eqref{ePlusBound}, this implies that
$$
\dim A_* = \dim A \leq \dim E_{+} (A) + 1 \leq \frac{n-1}{2} + 1.
$$
This would imply 
$
\dim A \leq \frac{n+1}{2},
$
which is a contradiction. Hence, $\dim A_* < \dim A$, which means that $\dim A_x = \dim A,$ and we have concluded the proof of Theorem \ref{thm-main}.(2).
\end{proof}

A small modification of the above argument gives our Fourier dimensional Proposition \ref{prop:fourierdimension}.

\begin{proof}[Proof of Proposition \ref{prop:fourierdimension}]
We assume that $A \subset \mathbb{R}^n$ satisfies $\dim A > 1$ and $\dim_F A \geq u$ for some $u \in (0,1]$. By Proposition \ref{prop-zeromeasureradproj}, the assumption regarding the Hausdorff dimension of $A$ guarantees that $\dim \pi_x (A) \geq s -1 > 0.$ We will next utilize the assumption on the Fourier dimension to show that
\begin{equation}\label{eq:emptyexceptionalset}
E_u (A) : = \{\theta \in \mathbb{S}^{n-1} : \dim \big(P_{\theta} (A)\big) < u \} = \emptyset.
\end{equation}
Since $\dim_F A \geq u$, for any $\epsilon > 0$, then by definition of Fourier dimension, there necessarily exists a measure $\mu \in \mathcal{M} (A)$ such that:
\begin{equation}\label{fMu}
\vert \hat{\mu} (\xi) \vert \lesssim \vert \xi \vert^{-\frac{u}{2} + \epsilon}.
\end{equation}
Recall that $\ell_\theta$ is the line through the origin with direction $\theta \in \mathbb{S}^{n-1}$. For each such $\theta$, let $\mu_{\ell_{\theta}} : = (\pi_{\theta})_{*} (\mu)$ denote the pushforward of this measure under the orthogonal projection map onto $\ell_\theta$. Then, $\mu_{\ell_\theta} \in \mathcal{M} \big(\pi_{\theta} (A) \big)$ and its Fourier transform is
$$
\widehat{\mu_{\ell_{\theta}}} (\xi) = \int_{\mathbb{R}^n} e^{-2 \pi i \xi \cdot \pi_{\theta} (x)} f(x) d \mu (x) = \int_{\mathbb{R}^n} e^{-2 \pi i \xi_{\theta} \cdot x} f(x) d \mu (x) = \hat{\mu} (\xi_{\theta}),
$$
where $\xi_{\theta}$ is defined as the unique point of intersection of the line $\ell_{\theta}$ and the fiber $\pi_{\theta}^{-1} (\xi)$. In particular, by \eqref{fMu}, we see that
$$
\vert \widehat{\mu_{\ell_{\theta}}} (\xi) \vert \lesssim \vert \xi \vert^{- \min \{\frac{1}{2}, \frac{u}{2} - \epsilon\}}.
$$
Since this holds for every $\epsilon > 0$, we have that
$$
\dim_F \big(\pi_{\theta} (A) \big) \geq \min \{1, u \}, \quad \forall \, \theta \in \mathbb{S}^{n-1}.
$$
In particular, since $\dim _F \pi_{\theta} (A) \leq \dim \pi_{\theta} (A)$, by Proposition \ref{fDimProp}, this proves \eqref{eq:emptyexceptionalset}. Together with the fact that $\dim \pi_x (A) > 0$, this shows that the set
$$
\overline{A}_x : = \pi_x (A) \setminus E_u (A),
$$
satisfies $\dim \overline{A}_x = \dim \pi_x (A) > 0$. The rest of the proof proceeds as in the proof of Theorem \ref{thm-main}.$(2)$ above.
\end{proof}

\subsection{Proofs of translation dependent results}

\begin{proof}[Proof of Theorem \ref{thm:main-translation}]
Here we prove Theorem \ref{thm:main-translation}.(2). As in the proof of Theorem \ref{thm-main}, the arguments of (1) and (3) are nearly identical, except with the different definitions of $E(A)$.

Fix $s > \frac{n}{2}$ and let $k = \lfloor \frac{n+1}{2} \rfloor$. Suppose $A,X \subset \mathbb{R}^n$ satisfy $\dim A = s$, $\dim X \geq s$, and $X$ has $k$-planar dispersion. Notice that by applying Proposition \ref{prop-measure_ESE} we have
\begin{equation}\label{eq-strictequalityexcep}
\dim E_{+} (A) \leq n - s < \frac{n}{2},
\end{equation}
where the last inequality is strict since $\dim A = s > \frac{n}{2}$.

We now apply Corollary \ref{cor:epsilonradproj} to $A$ and $X$. Let
$
\psi_n (A,X) = \min \{\dim X, \dim A, k\}.
$
Given, $\dim X \geq s = \dim A > \frac{n}{2}$, it follows that
$
\psi_n (A,X) \geq \frac{n}{2} > \dim E_+(A).
$
So, we define $X_A \subset X$ as
$$
X_A : = \{x \in X : \dim \pi_x \big(A \big) > n - s \}.
$$
Since $n-s \leq \psi_n(A,X) - \epsilon$ for some $\epsilon>0$, we know by Corollary \ref{cor:epsilonradproj} that $\dim (X_A) = \dim X$.

We now verify that for all $x\in X_A$, there exists a full dimensional subset $A_x\subset A$ such that $\mathcal H^1(\Pi_x^a(A))>0$ for all $a\in A_x.$ This verification is similar to that of Theorem \ref{thm-main}. Let $x \in X_A$, and define:
$$
\overline{A_x} : = \pi_x (A) \setminus E_{+}(A).
$$
By the definition of $X_A$, $\dim \pi_x(A) > \dim E_+(A)$ and thus we know that $\dim \overline{A_x} = \dim \pi_x (A) \geq \psi_n (A,X)$. Now, define the set $A_x$ as:
$$
A_x : = \pi_x^{-1} (\overline{A_x})\cap (A \setminus \{x\}) = \{a \in A \setminus \{x\} : \pi_x (a) \in \overline{A_x} \},
$$
which is chosen so that (by definition)
$$
\mathcal{H}^1 \big(P_{\pi_x (a)} (A) \big) > 0, \quad \forall a \in A_x.
$$
Applying the Key Lemma \ref{lma-keylemma} at each point $a \in A_x$ gives
$$
\mathcal{H}^1 \big(\Pi_x^a (A)\big) > 0, \quad \forall a \in A_x.
$$
\end{proof}

We finish with proof of Proposition \ref{prop:restricted} which assumes \textit{restricted geometric} structure on the set $\pi_x (A)$ (for some $x\in \R^n$). We will use a restricted exceptional set estimate of Zahl. Recall from the introduction that a \textit{non-degenerate curve} is a smooth map $\gamma : [0,1]\to \mathbb{S}^{n-1}$ such that for every $t \in [0,1]$, one has
\[
\det [\gamma(t), \gamma^{(1)}(t), \cdots , \gamma^{(n-1)}(t)] \neq 0.
\]
Then, Zahl's result in $\R^n$ (proven in $\R^2$ by Pramanik--Yang--Zahl \cite{pramanikyangzahl}) states:

\begin{theorem}[\cite{zahl2023maximalfunctionsassociatedfamilies}, Theorem 1.10] \label{thm:zahlresult}
    Let $A\subset \R^n$ and let $\gamma: [0,1]\to \mathbb{S}^{n-1}$ be a non-degenerate curve. Then, for $0< u \leq \min\{\dim A,1\}$,
    \[
    \dim \{t\in [0,1] : \dim P_{\gamma(t)}(A) < u\} \leq u.
    \]
\end{theorem}

We can now directly apply this result to obtain Proposition \ref{prop:restricted}.

\begin{proof}[Proof of Proposition \ref{prop:restricted}]
    By assumption, we have some $x\in \R^n$ such that there exists a non-degenerate curve $\gamma: [0,1]\to \mathbb{S}^{n-1}$ such that 
    \[
    \dim \{t \in [0,1] : \gamma(t) \in \pi_x(A)\} >u.
    \]
    Define $E_{u,\gamma}(A):=\{t\in [0,1] : \dim P_{\gamma(t)}(A) < u\}$. Applying \ref{thm:zahlresult} to this curve, we see that 
    \[
    \dim E_{u,\gamma}(A) < \dim \{t \in [0,1] : \gamma(t) \in \pi_x(A)\}.
    \]
    This implies that there exists a $u$-dimensional set of $t\in [0,1]$ such that direction $\gamma(t)$ lies in $\pi_x(A) \setminus E_{u,\gamma}(A)$. Hence, the same as in the proof of Theorem \ref{thm:main-translation}, one obtains that there exists a $u$-dimensional set of pins $A_x\subset A$ such that for all $a\in A_x$, there exists a $t\in [0,1]$ with $\pi_x(a) = \gamma(t)\in E_{u,\gamma}(A)^c$. The result then follows from the definition of $E_{u,\gamma}(A)$ and applying Key Lemma \ref{lma-keylemma}.
\end{proof}

%\begin{theorem}[\cite{ren2023furstenberg}, Theorem 1.2]
    %Let $A\subset \R^2$ be Borel and let $0 \leq u \leq \min\{\dim A, 1\}$. Then, $\dim E_u(A) \leq \max\{2u-\dim A, 0\}.$
%\end{theorem}

%\noindent This powerful result will be what allows us to obtain Theorem \ref{thm:mainrenwang}.

\section{Trees and sharpness constructions}\label{sec:trees}

While the focus of most of the results have thus far been about the abundance of single dot products determined by pairs of points, we now turn our attention to several results concerning many more dot products determined by many more points. By following the graph decomposition algorithm in \cite{OT2020} due to Yumeng Ou and Krystal Taylor, Theorem \ref{thm-main} gives the following result on the abundance of distinct types of dot product $k$-trees. To state the result, we first set some notation.

Given a tree $T$ with vertices $v_1, \dots, v_{k+1}$ and edges $\{v_{1,1},v_{1,2}\},\dots\{v_{k,1},v_{k,2}\},$ and some set $A\subset \mathbb R^n$, we write $x\mapsto v$ if the point $x\in A$ is being considered as a vertex $v\in T.$ Define the set of $k$-tuples of dot products between pairs of points in some $(k+1)$-tuple corresponding to the vertices of $T$ to be $\mathcal T_T(A).$ Here, we can think of the $(k+1)$-tuples as determining $k$ edge weights for an isomorphic copy of $T$. Specifically,
\[\mathcal T_T(A):=\{ (\alpha_1, \dots, \alpha_k): (x_1, \dots, x_{k+1})\in A^{k+1}, x_j \mapsto v_{i,1}, x_{j'} \mapsto v_{i,2}, x_j\cdot x_{j'} = \alpha_i\}.\]

For example, if we wanted to count how many triples of points chosen from a set $A\subset \mathbb R^n$ determine a pair of dot products, $(\alpha_1,\alpha_2)$, then we would have the tree $T$ be a 2-path with vertices $v_1, v_2, v_3,$ and edges $\{v_1,v_2\}$ and $\{v_2,v_3\}.$ We would consider all triples of points $(x_1,x_2,x_3)$ chosen from $A$ where $x_1\cdot x_2 = \alpha_1$ and $x_2\cdot x_3 = \alpha_2.$ In analogy to the pinned results above, we can also consider the pinned version of $\mathcal T_T(A),$ which we denote by $\mathcal T_{T,v}^x(A)$ which is the same as $\mathcal T_T(A)$ except that we mandate that $x \mapsto v$ for every $(k+1)$-tuple of points.

\begin{theorem}
Given a subset $A\subset \mathbb R^n,$ with $n\geq 2$ and $\dim A > \frac{n+1}{2},$ and a tree on $k$ vertices $T$ with vertex $v$, there exists a set $A_v\subset A$ so that for all $x\in A_v,$
\[\mathcal H^k\left(\mathcal T_{T,v}^x(A)\right)>0, \text{ for all } x\in A_v.\]
\end{theorem}

In a similar vein, we have two constructions with the measure of $(k+1)$-tuples of points corresponding to a dot product $k$-tree of a given type is large. One of the reasons such constructions have been studied is they have served to show when a given technique can go no further on its own. For example, the initial argument of Falconer from \cite{Falconer_1985} hinged on showing that the measure of the set of pairs of points that gave any fixed distance must be small. Specifically, he showed that if $s>\frac{n+1}{2}$, and $\mu$ was a Frostman measure on $\mathbb R^n,$
\[\int\int |x-y|^{-s}d\mu(x)d\mu(y)<\infty\Rightarrow (\mu\times\mu)\{(x,y):1\leq |x-y| \leq 1+\epsilon\}\lesssim \epsilon,\]
and used this to show that the Lebesgue measure of $\Delta(A)$ must be positive essentially by pigeonholing the mass of the product measure and taking the limit as $\epsilon \rightarrow 0.$ In \cite{mattila1987}, Mattila showed that this approach was tight in the sense that there are measures like $\mu$ above for $s<\frac{d+1}{2}$ that will fail the product measure estimate. That is, mass concentrates in the product measure too much for a pigeonholing scheme to work. This indicated that fundamentally new techniques had to be developed to lower the threshold below $\frac{n+1}{2}.$ As such, there are analogous constructions in \cite{EIK11}, \cite{IS20}, and \cite{IS16}. The constructions here follow techniques used in those papers as well as ideas from \cite{AGHS} and \cite{KMS}. Here, we suppose $\vec \alpha = (\alpha_1, \dots, \alpha_k)$ and introduce the notation $\mathcal T_{T}(A; \vec \alpha, \epsilon)$, which is the set of $(k+1)$-tuples of points in $A$ forming a dot product $k$-tree whose edge weights correspond to the vector $\vec \beta,$ where $|\beta_i-\alpha_i|<\epsilon$ for each entry.

\begin{theorem}\label{manyTrees}
Given a natural number $n\geq 2$, an $\epsilon>0,$ and a tree on $k$ edges, $T$, whose vertices can be two-colored with $W$ being its smallest color class, there exists a $k$-tuple $\vec \alpha,$ and a set $A\subset \mathbb R^n$ with $\dim A = s < \max\left\{n-1,\frac{n+1}{2}\right\}$ supporting a measure $\mu$ for which
\[\mu^{k+1}(\mathcal T_T(A;\vec\alpha, \epsilon))>\epsilon^{s(k+1)\frac{n-1}{n+1}}+\epsilon^{s|W|}.\]
\end{theorem}

We provide a sketch of the proof, as it follows by a series of minor modifications of rather detailed constructions explained completely in \cite{AGHS, KMS, IS20}. To begin recall some basic facts about trees. A vertex adjacent to exactly one other vertex in $T$ is called a leaf. It is well-known that any tree must have at least two leaves. It is also well-known that the vertices of any tree $T$ can be two-colored, or partitioned into two subsets such that no pair of vertices in either subset are adjacent.

\subsection{First exponent}
The bound $n-1$ comes from a simple construction based on the proof of Proposition 2.6 from \cite{AGHS}. We two-color the tree, and call the color classes $U$ and $W$. Without loss of generality, suppose that $|U|\geq |W|.$ Label the vertices from $U$ as $u_1, u_2, \dots, u_{|U|}.$ Next, label the neighbors of $u_1$ as $w_{1,1}, w_{1,2}, \dots,$ until all neighbors of $u_1$ have a label. Then label the unlabeled neighbors of $u_2$ as $w_{2,1}, w_{2,2}, \dots,$ until all neighbors of $u_1$ and $u_2$ have a label. Continue in this way, greedily labeling the unlabeled neighbors of successive $u_j,$ until all vertices have a label.

To construct the set $A$, pick $k$ parallel $(n-1)$-hyperplanes, intersect them with suitable fractal sets of Hausdorff dimension $s\leq (n-1)$ and call the resulting sets $h_i$. Now we show that this satisfies the conclusion of the theorem. Call the unique radial line that passes through all of the hyperplanes $\ell.$ Now, any $x_1$ in $h_1$ will have a dot product in the range $(\alpha_1+\epsilon,\alpha_1-\epsilon)$ with any point in $h_2$ within $n^{-n}\epsilon$ of $\ell$. So we select $x_1$ (corresponding to $u_1$) anywhere on $h_1.$ Then any point in this small neighborhood of $\ell$ on $h_2$ can serve as $x_2$ (corresponding to $w_{1,1}$). We continue, selecting the $x_j$ corresponding to vertices from $U$ anywhere in their respective intersected hyperplanes, and the $x_j$ corresponding to vertices from $W$ are chosen from their respective intersected hyperplanes, but within $n^{-n}\epsilon$ of $\ell$. Putting this all together, we get
\[\mu^{k+1}(\mathcal T_T(A;\vec\alpha, \epsilon))>\epsilon^{s|V|}.\]

\subsection{Second exponent}

The bound $\frac{n+1}{2}$ comes from the measure constructed in Section 2.2 of \cite{IS20}. The basic idea there is to construct two subsets of $\mathbb R^n,$ called $E$ and $F$, each with a prescribed Hausdorff dimension $s<\frac{n+1}{2}.$ For any $\epsilon>0,$ these sets have the property that for a positive proportion of the $x\in E$, the measure of the set of $y\in F$ whose dot product is within $\epsilon$ of 1 is substantial. In the end, we will set $A:=E\cup F.$

With the sets $E$ and $F$ in tow, we proceed by induction on $k.$ In the base case, when $k=1,$ we simply appeal to Theorem 1.2 from \cite{IS20}. Now for the inductive step. If $k>1,$ then fix a leaf $v$ from the tree $T$. Define $T'$ to be the tree $T$ without the vertex $u$ and the edge $\{u,w\}$, where $w$ is some other vertex in $T$. Clearly $T'$ is a tree on $(k-1)$ edges. By our inductive hypothesis, the set $A$ satisfies
\[\mu^{k}(\mathcal T_T(A;\vec\alpha, \epsilon))>\epsilon^{sk\frac{n-1}{n+1}}.\]
Without loss of generality, suppose that at least half of the measure above is contributed by $k$-tuples of points whose point $y$ corresponding to the vertex $w$ comes from the set $F.$ Now we just need to note that a positive proportion of these points $y\in F$ have substantial measure's worth of points $x\in E$ so that $x\cdot y\in(1-\epsilon,1+\epsilon).$

\printbibliography
                
\end{document}